\documentclass[11pt,  A4paper]{amsart}
\usepackage{amsfonts,amssymb,amsmath,amsthm,amscd,mathtools,multicol,tikz, tikz-cd,caption,enumerate,mathrsfs,geometry,tgpagella}
\usetikzlibrary{arrows,chains,matrix,positioning,scopes}
\usepackage{inputenc}
\usepackage[foot]{amsaddr}

\usepackage[pagebackref=true, colorlinks, linkcolor=blue, citecolor=red]{hyperref}
\usepackage{fullpage}
\usepackage[all]{xy}

\makeindex

\newtheorem{theorem}{Theorem}[section]
\newtheorem{proposition}[theorem]{Proposition}
\newtheorem*{theorem*}{Theorem}
\newtheorem{corollary}[theorem]{Corollary}
\newtheorem{lemma}[theorem]{Lemma}
\newtheorem{remark}[theorem]{Remark}
\theoremstyle{remark}
\newtheorem{example}[theorem]{Example}

\numberwithin{equation}{section}

\newcommand{\bt}{\begin{theorem}}
	\newcommand{\et}{\end{theorem}}
\newcommand{\bts}{\begin{theorem*}}
	\newcommand{\ets}{\end{theorem*}}
\newcommand{\bco}{\begin{corollary}}
	\newcommand{\eco}{\end{corollary}}
\newcommand{\bd}{\begin{definition}}
	\newcommand{\ed}{\end{definition}}
\newcommand{\bp}{\begin{problem}}
	\newcommand{\ep}{\end{problem}}
\newcommand{\bl}{\begin{lemma}}
	\newcommand{\el}{\end{lemma}}
\newcommand{\bprop}{\begin{proposition}}
	\newcommand{\eprop}{\end{proposition}}
\newcommand{\br}{\begin{remark}}
	\newcommand{\er}{\end{remark}}
\newcommand{\bpf}{\begin{proof}}
	\newcommand{\epf}{\end{proof}}
\newcommand{\bex}{\begin{example}}
	\newcommand{\eex}{\end{example}}

\newcommand{\Z}{\mathbb{Z}}
\newcommand{\Q}{\mathbb{Q}}

\newcommand{\gal}{\mathscr{G}}

\begin{document}
	
	\title{Liouville's Theorem on integration in finite terms for $\mathrm D_\infty,$ $ \mathrm{SL}_2$  and Weierstrass field extensions}
	\author{Partha Kumbhakar \and Varadharaj R. Srinivasan}
	\email{parthakumbhakar@iisermohali.ac.in, \ \ ravisri@iisermohali.ac.in}
	\address{Indian Institute of Science Education and Research (IISER) Mohali, Knowledge City,  Sector 81, S.A.S. Nagar 140306, Punjab, India}

	\maketitle
	
	\begin{abstract} Let $k$ be a differential field of characteristic zero  and the field of constants $C$ of $k$ be an algebraically closed field.  Let $E$ be a differential field extension of $k$ having $C$ as its field of constants and that $E=E_m\supseteq E_{m-1}\supseteq\cdots\supseteq E_1\supseteq E_0=k,$ where $E_i$ is either an elementary extension of $E_{i-1}$ or $E_i=E_{i-1}(t_i, t'_i)$ and $t_i$ is weierstrassian (in the sense of Kolchin \cite[Page 803]{Kolchin1953}) over $E_{i-1}$ or $E_i$ is a Picard-Vessiot extension of $E_{i-1}$ having a differential Galois group isomorphic to either the special linear group $\mathrm{SL}_2(C)$ or the infinite dihedral subgroup $\mathrm{D}_\infty$ of $\mathrm{SL}_2(C).$ In this article,  we prove that Liouville's theorem on integration in finite terms \cite[Theorem]{Rosenlicht1968} holds for $E$. That is,  if $\eta\in E$ and $\eta'\in k$ then there is a positive integer $n$ and for $i=1,2,\dots,n,$ there are elements $c_i\in C,$ $u_i\in k\setminus \{0\}$ and $v\in k$ such that $$\eta'=\sum^n_{i=1}c_i\frac{u'_i}{u_i}+v'.$$ 
	\end{abstract}

	\section{Introduction} 
	Let $k$ be a differential field of characteristic zero with the derivation $x\mapsto x'$ and $C_k:=\{x\in k\ | \ x'=0\}$ be the field of constants of $k.$ Let $E$ be a differential field extension of $k.$  An element $\theta\in E$ is said to be \emph{elementary} (respectively, \emph{liouvillian}) over $k$ if either $\theta$ is algebraic over $k$ or for some $\alpha\in k\setminus\{0\},$ $\theta'=\alpha'/\alpha$ (respectively, $\theta'=\alpha$) or   for some $\alpha\in k,$ $\theta'/\theta=\alpha'$ (respectively, $\theta'/\theta=\alpha$).  
	A differential field extension $E$ of $k$ is called  an \emph{elementary field extension} (respectively, a \emph{liouvillian field extension})  of $k$  if  there is a tower of differential fields
	\begin{equation*}
		k=E_0\subseteq E_1\subseteq \cdots\subseteq E_m=E
	\end{equation*}
	such that for each $i=1,\dots, m,$ we have $E_i=E_{i-1}(\theta_i)$ for some $\theta_i$ elementary (respectively, liouvillian) over $k.$

An element $f\in k$ is said to have an \emph{elementary integral} over $k$ if there are constants $c_1,\dots, c_n\in C_k$ and nonzero elements $u_1,\dots, u_n\in k$ and an element $v\in k$ such that 	
\begin{equation*}\label{elementaryexpression} c_1u'_1/u_1+\cdots+c_nu'_n/u_n+v'=f.\end{equation*} 
Observe that if $f\in k$ has an elementary integral over $k$  then one can construct an elementary extension field $E$ such that $E=k(x_1,\dots,x_n)$ and that for each $i,$ $x'_i=u'_i/u_i.$ Here, one can think $x_i$  as ``$\log(u_i)$" and that $\int f=\sum^n_{i=1}c_i\log(u_i)+v.$

In \cite{Rosenlicht1968}, Rosenlicht proved the following theorem, which is known as Liouville's Theorem: Let $k$ be a differential field of characteristic zero and $E$ be an elementary extension of $k$ with $C_E=C_k$. Suppose that there is an element $f \in k$ 
having an elementary integral over $E.$ Then $f$ has an elementary integral over $k.$ It is then natural to ask for an extension of  Liouville's theorem when the differential field $E$ is built up  by successive adjunction of elements that are not necessarily elementary over the predecessor field.  In fact, there are several such extensions of Liouville's theorem available.  For example, in  \cite{SSK1985}, \cite{Badd2006}, \cite{Hebisch2018} and \cite{YKVRS2019} the nonelementary adjunctions considered are by indefinite integrals  such as error functions,  logarithmic integrals and polylogarithmic integrals. 

In this article, we prove yet another extension of Liouville's theorem. The differential field extensions  we consider are obtained by successive adjunction of elements that are either elementary or  solutions of certain second order differential equations or solutions of Weierstrass differential equations, over the predecessor field. The precise statement of our theorem is as follows.
	\begin{theorem}\label{IFT-elliptic}\label{IFT-euler}
		Let $k$ be a differential field of characteristic zero and $C_k$ be an algebraically closed field. Let $$E=E_m\supseteq E_{m-1}\supseteq \cdots\supseteq E_1\supseteq E_0=k$$ be a chain of differential fields such that $C_E=C_k$ and for each $i=1,2,\dots, m,$ $E_i$ is of one of the following types:
		\begin{enumerate}[(i)]
		\item\label{elementary} $E_i$ is an elementary extension of $E_{i-1}.$\\
		\item\label{eulerian}  $E_i$ is a Picard-Vessiot extension of $E_{i-1}$ having a differential Galois group isomorphic to $\mathrm{SL}_2(C_k)$ or to the dihedral subgroup $\mathrm{D}_\infty$ of $\mathrm{SL}_2(C_k)$  $$\mathrm{D}_\infty:=\left\{\begin{pmatrix}a&0\\0&a^{-1}\end{pmatrix}\ \Big| \ a\in C_k\setminus \{0\}\right\}\bigcup \left\{\begin{pmatrix}0&a\\ -a^{-1}&0\end{pmatrix}\ \Big| \ a\in C_k\setminus \{0\}\right\}.$$  
		\item \label{weierstrassianelement}$E_i=E_{i-1}(\theta, \theta'),$ where $\theta$ is \emph{weierstrassian}\footnote{Definition is due to Kolchin \cite[p. 803]{Kolchin1953}} over $E_{i-1}.$  That is,  $\theta$ is transcendental over $E_{i-1}$  and there are a polynomial $P(X)=4X^3-g_1X-g_0\in C_k[X]$ with a non-zero discriminant\footnote{That is, $27g_0^2-g^3_1\neq 0,$ which then implies that the Weierstrass elliptic curve $ZY^2-4X^3+g_1Z^2X+g_0Z^3$  is non-singular.} and a nonzero element $\alpha\in E_{i-1}$ such that   $$\theta'^2=\alpha^2 P(\theta)\quad(\emph{Weierstrass differential equation}).$$ 
		\end{enumerate}
		
		Suppose that $f\in k$ has an elementary integral over $E.$ Then $f$ has  an elementary integral over $k.$ 
	\end{theorem}

 What is special about our theorem is that our extension field $E,$ unlike the extension fields considered in \cite{SSK1985}, \cite{Badd2006}, \cite{Hebisch2018},  need not be a liouvillian extension field:  we  allow  adjuction of weierstrassian elements and solutions of second order differential equations with Galois groups  isomorphic to $\mathrm{SL}_2(C_k).$ 
 
In the manuscript \cite{Hebisch21}, Hebisch considered  \emph{elliptic-Lambert} field extensions, which are obtained by a repeated adjunction of  elliptic functions, Lambert functions and elliptic integrals. Wherein, for the field extensions obtained by a  repeated adjunction of  elementary and elliptic functions, he proved that Liouville's theorem holds.  His definition of an elliptic function $\eta$ over a field $k$ is  that   \begin{equation}\label{hebisch-elliptic} \eta'^2=\beta'^2 P(\eta),\end{equation} where $\beta\in k$ and  $P(X)=X^3-g_1X-g_0\in C_k[X].$ 
 In  Theorem \ref{IFT-elliptic} (\ref{weierstrassianelement}), the element $\alpha\in E_{i-1}$ that appears in the Weierstrass differential equation is arbitrary and in particular, $\alpha$ need be of the form $\beta'$ for any $\beta\in E$ or for any $\beta$ from an elliptic-Lambert field extension of $E_{i-1}.$   Therefore, weierstrassian elements at large are not covered in \cite{Hebisch21}. However, if $\eta$ is an elliptic function, as in Equation (\ref{hebisch-elliptic}), over $E_{i-1}$  then it is clearly  weierstrassian over $E_{i-1}$.  Thus, the extension field $E$ considered in our theorem does include the adjunction of Hebisch's elliptic functions.

	\section{Preliminaries and Basic Results} \label{prelims}
	
In this paper,  by a differential field, we  mean a field of characteristic zero with a single derivation map. We fix a differential field $k$ and assume that the field of constants $C$ of $k$ is an algebraically closed field.  We shall now record few basic results from differential algebra that are needed in our proofs.
	
	 \subsection{Second order differential equations}
	 Let $E$ be a Picard-Vessiot extension of $k$ for a differential $k-$module $M.$ Suppose that the differential Galois group $\gal(E/k)$ is  isomorphic to a closed subgroup of $\mathrm{SL}_2(C)$.  Then $\gal(E|k)$ is isomorphic to  one of the following groups \cite[Page 7, Lemma]{Kovacic1986}.
	 
	  	\begin{enumerate}[(i)]	
	  			  	\item\label{finitesubgroup} a finite group.
	  
	  	\item\label{dihedralsubgroup}  the infinite dihedral subgroup of $\mathrm{SL}_2(C);$ $$\mathrm{D}_\infty:=\left\{\begin{pmatrix}a&0\\0&a^{-1}\end{pmatrix}\ \Big| \ a\in C\setminus \{0\}\right\}\bigcup \left\{\begin{pmatrix}0&a\\ -a^{-1}&0\end{pmatrix}\ \Big| \ a\in C\setminus \{0\}\right\}.$$  

	  	\item\label{sl2case} $\mathrm{SL}_2(C).$
	  	
	  		\item \label{borel}  a closed subgroup of a Borel subgroup  of $\mathrm{SL}_2(C).$ 
	  \end{enumerate}
	 
	 In the next two propositions, we explain known structure properties of Picard-Vessiot extensions whose Galois groups are isomorphic to the infinite dihedral group or the special linear group. The results obtained will be used in the proof of our main theorem. 
	 
	 \bprop \label{dihedralanalysis}
	 	Let $E$ be a Picard-Vessiot extension of $k$ with Galois group $\gal(E|k)$ isomorphic to $\mathrm{D}_\infty$ as algebraic groups. Then there is a tower of differential fields $$k\subsetneqq k(\alpha)\subsetneqq E=k(\alpha)(\eta)$$ having the following properties
	 	\begin{enumerate}[(i)] \item $k(\alpha)$ is a quadratic extension of $k$ as well as  the algebraic closure of $k$ in $E.$  	\\ 	 \item The element $\eta$ is transcendental over $k(\alpha)$ with $\eta'/\eta=\alpha.$ \\
	 		\item There is an element $\gamma\in k\setminus \{0\}$  such that the trace $$\mathrm{tr}(\alpha)=\frac{1}{2}\frac{\gamma'}{\gamma}.$$ 
	 	\end{enumerate}
\eprop

\begin{proof}
	Since $\gal(E|k)=\mathrm D_\infty\cong \mathrm G_m\rtimes \Z_2,$ the identity component $\gal(E|k)^0$ is isomorphic to $\mathrm G_m$ and the quotient group $\gal(E|k)/\gal(E|k)^0$ is isomorphic to $\Z_2.$  Therefore,  from the fundamental theorem of differential Galois theory and from \cite[Lemma A1]{Rubel-Singer1988}, there is a tower of differential fields \begin{equation*} k\subseteq k(\alpha)\subseteq E=k(\alpha)(\eta),\end{equation*} where $k(\alpha)$ is the algebraic closure of $k$ in $E,$  $k(\alpha)$ is a quadratic extension of $k,$ $\eta$ is transcendental over $k(\alpha)$ and that $\eta'/\eta\in k(\alpha)\setminus k.$ Furthermore, \begin{align*}\gal(k(\alpha)|k)& \cong \gal(E|k)/\gal(E|k)^0\cong \Z_2\\  \gal(E|k(\alpha))&\cong \gal(E|k)^0\cong \mathrm G_m.\end{align*} Since, $k\subsetneqq k(\eta'/\eta)
	 \subseteq k(\alpha)$ and that $[k(\alpha):k]= 2,$ we have $k(\eta'/\eta)=k(\alpha).$  Therefore, we may assume $\alpha=\eta'/\eta.$  Let $\alpha$ and $\beta$ be the distinct roots of the irreducible polynomial of $\alpha.$ Then, there is an automorphism $\tau\in \gal(E|k)$ such that $\tau(\alpha)=\beta.$ Let  $\tau(\eta)=:\zeta$ and observe that  $\zeta'/\zeta=\beta$  and that \begin{equation}\label{productofsolutions}\frac{(\eta\zeta)'}{\eta\zeta}=\alpha+\beta\in k.\end{equation} 
	
	Since $\gal(E|k)\cong \mathrm D_\infty$ has no closed normal subgroup $N$ such that $\mathrm D_\infty/N \cong \mathrm G_m,$ it follows that $\eta\zeta$ is not transcendental over $k.$ Thus $\eta\zeta$ belongs 
	to the quadratic extension $k(\alpha).$   Now from Equation (\ref{productofsolutions}) and from \cite[Remark 1.11.1]{Magid1994}, we obtain $(\eta\zeta)^2\in k.$  Let $\gamma:=(\eta\zeta)^2$ and observe that
	\begin{equation*} \frac{(\eta\zeta)'}{\eta\zeta}=\alpha+\beta=\mathrm{tr}(\alpha)=\frac{1}{2}\frac{\gamma'}{\gamma}.
		\end{equation*}	\end{proof}
	
	\bprop \label{SLanalysis} \cite[page 58]{Magid1994}
	Let $E$ be a Picard-Vessiot extension of $k$ with Galois group $\gal(E|k)$ isomorphic to $\mathrm{SL}_2(C)$ as algebraic groups. Then we have the following:
	\begin{enumerate}[(i)]
		\item\label{existencesecondorder} $E$ is a Picard-Vessiot extension of $k$ for a matrix differential equation  $y'=Ay,$ where  $$A=\begin{pmatrix}0&1\\r &s\end{pmatrix}\in M_2(k).$$
		\item\label{structureSL2} there is a tower of differential fields $$k\subsetneqq k(\alpha)\subsetneqq k(\alpha, \xi)\subsetneqq k(\alpha, \xi,\eta)=E,$$ 
where $\alpha, \xi$ and $\eta$ are $k-$algebraically independent,
$\eta'=\omega/\xi^2$ for some $\omega\in k\setminus \{0\},$  $\xi'=\alpha \xi$ and that $\alpha$ is a zero of the Riccati differential polynomial $R(X)=X'+X^2-rX-s.$ \\
\item \label{noricattisolnink} $R$ has no zeros in $k.$
	\end{enumerate}
	\eprop
	
\begin{proof} Let $R$ be the Picard-Vessiot ring of $E.$  Then by \cite[Corollaries 5.17 and 5.29]{Magid1994}, there is an isomorphism of $k-$algebras 
	\begin{equation*}\phi: R\to k\otimes_C C[\gal(E|k)], \end{equation*}
	where $C[\gal(E|k)]$ is the coordinate ring of $\gal(E|k),$  which  is also compatible with the 
action of $\gal(E|k).$  Since $$C[\gal(E|k)]=\frac{C[x_{11},x_{12}, x_{21},x_{22}]}{\langle x_{11}x_{22}-x_{21}x_{12}-1\rangle},$$ there are elements $y_1, y_2, y_3,y_4\in R$ (namely, the images of $x_{ij}$ under $\phi^{-1}$) that generates $R$  as a $k-$algebra.  Let $$Y:=\begin{pmatrix}y_1 &y_2\\ y_3 & y_4\end{pmatrix}\qquad W:=\begin{pmatrix}y_1 &y_2\\ y'_1 & y'_2\end{pmatrix}.$$ Then for any $\sigma\in \gal(E|k),$ we have \begin{equation}\label{sigmaY}\sigma(Y)=YC_\sigma\end{equation} for some $C_\sigma:=(c_{ij\sigma})\in \mathrm{SL}_2(C).$ In particular, for $i=1,2$ and for all $\sigma\in \gal(E|k)$,  $\sigma(y_i)=c_{1i\sigma}y_1+c_{2i\sigma}y_2$  and therefore  we also have \begin{equation}\label{sigmawronsk}\sigma(W)=WC_\sigma.\end{equation}

From Equations (\ref{sigmaY}) and (\ref{sigmawronsk}), we have the following observations. First, we see that  the entries of the matrices  $W'W^{-1}$ and $YW^{-1}$ are fixed by $\gal(E|k).$ Therefore $W'=AW$ and $Y=BW,$ for some $A,B \in M_2(k).$ Next,  $\sigma(\mathrm{det} (W))=\mathrm{det} (W)$ for all $\sigma\in \gal(E|k)$ and we obtain that  $\mathrm{det} (W)\in k.$  Now since entries of $Y$ generate $R$, so does the entries of $W.$  Finally, since  $\mathrm{tr.deg} (E|k)=3,$  the set $\{y_1,y_2,y'_1\}$  must be $k-$algebraically independent.  

Let $A=(a_{ij})\in M_2(k).$ Then $y'_1=a_{11}y_1+a_{12}y'_1$ and therefore $a_{11}=0$ and $a_{12}=1.$This proves (\ref{existencesecondorder}). 

Now from the equation $y'=Ay,$ we obtain a second order differential equation \begin{equation}\label{secordeqn}z''=a_{21}z'+a_{22}z\end{equation}
to which  $y_1, y_2$ are  $k-$algebraically independent solutions.  Let $\xi:=y_1,$  $\alpha=\xi'/\xi$ and $\eta= y_2/y_1$ and observe that these choices have the desired properties listed in (\ref{structureSL2}). 

Let $V$ be the set of all solutions in $E$ of the Equation (\ref{secordeqn}) and $\mathscr R$ be the set of all zeros in $E$ of the Riccati differential polynomial $R(X)=X'+X^2-a_{21}X-a_{22}$. Then, the logarithmic derivative map $x\to x'/x$ from $V\setminus \{0\}\to \mathscr R$ is surjective. If $u\in k\cap \mathscr R$ then since $k$ is algebraically closed in $E,$  there is a $v\in V$ that is transcendental over $k$ such that $v'/v=u.$ Note that $\sigma(v)=c_\sigma v$ for some $c_\sigma\in C\setminus \{0\}.$ Therefore, the field $k(v)$ is then a differential field  that is also stable under the action of $\gal(E|k).$ But then, by fundamental theorem of differential Galois theory, the closed subgroup $\gal(E|k(v))$ must be a normal subgroup of $\gal(E|k)\cong \mathrm{SL}_2(C)$ of dimension $2$. This contradicts the fact that $\mathrm{SL}_2(C)$ is a simple algebraic group. Thus, $k$ contains no zeros of  $R.$
\end{proof}

	\subsection{Weierstrassian elements}   Let $k\subset K$ be differential fields of characteristic zero such that $C_K=C$ and that $C$ be an algebraically closed field.  An element $\theta\in K$ is said to \emph{weierstrassian} over $k$ if $\eta$ is transcendental over $k$ and there is a nonsingular irreducible projective curve of genus $1$, defined over $C$, in the Weierstrass form:  $$\mathcal E:=ZY^2-4X^3+g_1Z^2X+g_0Z^3$$ such that for some $\alpha\in k\setminus \{0\},$ $(\theta: \theta'/\alpha:1)$ is a $K-$point of the curve. That is \begin{equation*} \label{weierstrassdefn}\theta'^2=\alpha^2(4\theta^3-g_1\theta-g_0);\quad \alpha\in k\setminus\{0\}, \quad g_0, g_1\in C, \quad27g^2_0-g^3_1\neq 0.
		\end{equation*}

	 In \cite[page 809]{Kolchin1953},   it was  proved that if $K$ is a  strongly normal extension of $k$ such that  $k$ is  algebraically closed in $K$ and  the field transcendence degree tr.d$(K|k)=1$  then there is an element $\theta\in K$ such that either $K=k(\theta)$ and $\theta'\in k$ or $K=k(\theta)$ and $\theta'/\theta\in k$ or $K=k(\theta,\theta')(\zeta)$ and $\theta$ is weierstrassian over $k$ and $\zeta$ is algebraic over $k(\theta,\theta')$. 
	 
	 We need the following facts about the differential $k-$automorphism group $\gal$ of the differential field  $k(\theta, \theta').$ The $C-$rational points of $\mathcal E,$ denoted by $\mathcal E(C),$ is a commutative group.  A point $p\in \mathcal E(C)$ induces a translation isomorphism  $\tau_p: \mathcal E(C)\to \mathcal E(C)$ of curves, defined by $\tau_p(x)=x+p.$ The function field $k(\mathcal E)$ of $\mathcal E(k)$ is isomorphic to $k(\theta,\theta')$ and  the map $\tau_p,$ for $p\in \mathcal E(C),$  induces a  differential $k-$automorphism  $\tau^*_p$ of  $k(\theta,\theta').$  Conversely, every differential $k-$automorphism $\tau$ of $k(\theta,\theta')$ also induces a  translation map $\tau_p,$ for some $p\in \mathcal E(C).$ The mapping $\varphi: \mathcal E(C)\to \gal$ defined by $\varphi(\tau_p)=\tau^*_p$ is in fact an isomorphism of commutative groups. There is a bijective Galois correspondence between the intermediate differential subfields of $k(\theta,\theta')$ and the closed subgroups (which are precisely the finite subgroups) of $\mathcal E(C).$ In particular, it can be shown that $k$ is algebraically closed in $k(\theta,\theta').$  We refer the reader to  \cite[pages 803-807]{Kolchin1953} or \cite[Example 2.7]{Kovacic2006} for a proof of these facts.

	Let $k\subset K$ be fields and  $\Omega_{K/k}$ be  the $K-$vector space of $k-$differentials. Note that $\Omega_{K/k}$ has the following universal property:  there is a $k-$derivation map $\mathrm{d}: K\to \Omega_{K/k},$ which by definition is $k-$linear and for $x,y\in K,$ $\mathrm{d}(xy)= x \ \mathrm{d}y+y \ \mathrm{d}x,$ such that for any $k-$derivation $D$ from $K$ to a $K-$vector space $M,$ there is a unique $K-$homomorphism $\phi: \Omega_{K/k}\to M$ such that $\phi\circ \mathrm{d}=
	D.$

	For the convenience of the reader, in the next few propositions, we reproduce certain results from \cite{Rosenlicht1968} and \cite{Rosenlicht1976} that will be used in the proof of our theorem.
	
	\begin{proposition}\cite[Proposition 4]{Rosenlicht1976} \label{zerodifferential}
		Let $k\subset K$ be fields of characteristic zero. Let $c_1,\dots, c_n$ be elements of $k$ that are linearly independent over the rational numbers $\Q\subset k$ and  $u_1,\dots, u_n$ be nonzero elements of $K$ and $v$ be an element of $K.$ Then the element $$c_1\frac{\mathrm{d}u_1}{u_1}+c_2\frac{\mathrm{d}u_2}{u_2}+\cdots+c_n\frac{\mathrm{d}u_n}{u_n}+\mathrm{d}v$$ of $\Omega_{K/k}$ is zero if and only if each $u_1,\dots,u_n, v$ is algebraic over $k.$ 
	\end{proposition}

	\begin{proposition} \cite[Lemma]{Rosenlicht1976}\label{differentialdependence} Let $k$ be a differential field of characteristic zero,  $K$ be a differential field extension of $k,$  $C_K=C$ and   tr.d$(K|k)=1.$  If there are two $k-$differentials of the form $c_1\mathrm{d}u_1/u_1+\cdots+c_n\mathrm{d}u_n/u_n+\mathrm{d}v,$ where $c_1,\dots, c_n$ are constants and each $u_1,\dots,u_n, v$ in $K$ such that    $c_1u'_1/u_1+\cdots+c_nu'_n/u_n+v'\in k$ then the differentials are linearly dependent over $C.$ 
	\end{proposition}

   \begin{proposition}\cite[Lemma]{Rosenlicht1968}\label{elementary}
   	Let $k$ be a differential field of characteristic zero. Let $k(t)$ be a differential field extension of $k$, $C_{k(t)}=C$ with $t$ is transcendental over $k$ and either $t' \in k$ or $t'/t \in k$. Let $c_1,\dots,c_n \in k$ be linearly independent over $\mathbb{Q}$ and let $u_1,\dots, u_n$ be nonzero element in $k(t)$, $v \in k(t)$. Then if 
   	\begin{equation*}
   		\sum_{i=1}^{n}c_i\frac{u_i'}{u_i} + v' \in k[t]
   	\end{equation*}
   we have $v \in k[t] $ and in the case $t' \in k$, each $u_i \in k$, while in the case $t'/t \in k$, for each $i=1,\dots, n$, $u_i/t^{\nu_i} \in k$ for some integer $\nu_i$.
   \end{proposition}

	\section{Main Results}

\begin{lemma} \label{dihedral-lemma}Let $E$ be a  Picard-Vessiot extension of $k$ having a differential Galois group isomorphic to the infinite dihedral group $\mathrm D_\infty.$  If $f\in k$ has an elementary integral over $E$ then it has an elementary integral over $k.$
	
\end{lemma}

	\begin{proof}Let  $c_1u'_1/u_1+\cdots+c_nu'_n/u_n+v'=f\in k,$ where $c_i\in C$ and $v, u_i\in E.$ We may further assume that $c_1,\dots, c_n$ are $\Q-$linearly independent (see \cite[page 158]{Rosenlicht1968}).  By Proposition \ref{dihedralanalysis} we have a tower $k \subset k(\alpha) \subset k(\alpha,\eta)=E$, where $k(\alpha)$ is a quadratic extension,  $\mathrm{tr}(\alpha)=\frac{1}{2}\frac{\gamma'}{\gamma},$ $\eta$ is transcendental over $k$ and $\eta'/\eta=\alpha.$  Now we use Proposition \ref{elementary} and  obtain that for each $i,$ $u_i=a_i\eta^{m_i}$ with $a_i \in k(\alpha),$  $m_i$ an integer, and $v \in k(\alpha)[\eta]$. Since $u_i'/u_i=a_i'/a_i + m_i \eta'/\eta=a_i'/a_i + m_i\alpha$. We have
		\begin{equation*}
			f = \sum_{i=1}^{n}c_i\frac{a_i'}{a_i} + (\sum_{i=1}^n m_ic_i)\alpha +v'.
		\end{equation*}
	    This implies $v' \in k(\alpha).$ Applying Kolchin-Ostrowski, \cite[Appendix]{Rubel-Singer1988},  we have $v \in k(\alpha).$  Hence
	   \begin{equation}\label{dihedral-exponential}
		f = \sum_{i=1}^{n}c_i\frac{a_i'}{a_i} + c\alpha +v'.
	   \end{equation}
       where $a_1,\dots,a_n,v \in k(\alpha)$ and $c=\sum_{i=1}^{n} m_ic_i$. 
       
       Let $G$ be the Galois group of the quadratic extension $k(\alpha)$ of $k.$ Then, for all $x\in k(\alpha),$ $$\mathrm{tr}(x)=\sum_{\sigma
       	\in G}\sigma(x)\in k, \qquad  \mathrm{nr}(x)=\prod_{\sigma\in G}\sigma(x)\in k\quad \text{and} \quad \mathrm{tr}(x)'=\mathrm{tr}(x').$$  Applying the trace map to Equation (\ref{dihedral-exponential}), we obtain \begin{align*}2f &= \sum_{i=1}^n c_i \mathrm{tr}\left(\frac{a_i'}{a_i}\right) + c \ \mathrm{tr}(\alpha) +\mathrm{tr}(v')\\ 
       &=\sum_{i=1}^{n}c_i \left(\frac{\mathrm{nr}(a_i)'}{\mathrm{nr}(a_i)}\right) + c\ \mathrm{tr}(\alpha) +\mathrm{tr}(v)'.\end{align*}
  Since we have proved in Proposition \ref{dihedralanalysis} that  $\mathrm{tr}(\alpha)=\frac{1}{2}\frac{\gamma'}{\gamma}$ for some $\gamma\in k,$  the proof of the Lemma is now complete.     \end{proof}
  
\begin{lemma} \label{sl_2-lemma} Let $E$ be a  Picard-Vessiot extension of $k$ having a differential Galois group isomorphic to $\mathrm{SL}_2(C).$  If $f\in k$ has an elementary integral over $E$ then it has an elementary integral over $k.$
	
\end{lemma}

	\begin{proof} 
		
We first  resolve $E$ as a tower of differential fields as  in Proposition \ref{SLanalysis} (\ref{structureSL2}). Then for $y:=\omega/\xi^2,$ we have  $$\beta:=\frac{y'}{y}=\frac{\omega'}{\omega}-2\alpha\in k(\alpha)$$ and therefore, $k(\alpha)=k(\beta)$ and that $k(\alpha, y)$ is a differential field. Since $\eta'=y$ and that $\xi^2\in k(\alpha, y),$   we have  the following tower of differential fields: $$k\subsetneqq k(\alpha)\subsetneqq k(\alpha, y)\subsetneqq k(\alpha, y,\eta)\subsetneqq E=k(\alpha, y,\eta)(\xi),$$ where $\alpha, \xi$ and $\eta$ are as in Proposition \ref{SLanalysis} (\ref{structureSL2})  and $E$ is a quadratic  extension of $k(\alpha, y,\eta)$ with $\xi^2\in k(\alpha,y).$  We divide the rest of the proof into four steps.
	
\emph{Step 1.} Let \begin{equation*} c_1 \frac{u_1'}{u_1} + c_2 \frac{u'_2}{u_2} + \dots + c_n\frac{u_n'}{u_n}+ v' = f,\end{equation*} where  $u_1,\cdots, u_n\in E$ and $c_1,\cdots, c_n$ are $\Q-$linearly independent constants  and $v \in E.$  Then since $E$ is an elementary (algebraic) extension of $k(\alpha, y,\eta),$ by Liouville's theorem, we obtain that $f$ admits a similar expression over $k(\alpha, y,\eta).$ Thus, we shall further assume that for each $i,$ $u_i\in k(\alpha, y, \eta)$ and that $v\in k(\alpha, y, \eta).$ 
	
\emph{Step 2.} Observe that $f\in k\subset k(\alpha,y)[\eta]$ and that $\eta'=y\in k(\alpha, y).$ Therefore, we shall apply Proposition \ref{elementary} and obtain that  $u_1,\dots, u_n \in k(\alpha,y)$ and that  $v\in k(\alpha,y)[\eta]$.  Since $v' \in k(\alpha,y),$ by Kolchin-Ostrowski,  there is a constant $e_1$ such that $v-e_1\eta \in k(\alpha, y).$ Thus    
		\begin{equation}\label{reduction-antiderivatives}
			f= \sum_{i=1}^{n}c_i\frac{u_i'}{u_i} + e_1y+(v-e_1\eta)'.   
		\end{equation}
		where $u_i,\dots u_n, v-e_1\eta \in k(\alpha,y)$.  
		
	\emph{Step 3.} We have $\sum^n_{i=1}c_iu'_i/u_i+(v-e_1\eta)'=f-e_1y\in k(\alpha)[y],$ where $y'=\beta y.$   Therefore, we shall again apply Proposition \ref{elementary} and obtain that $v-e_1\eta\in k(\alpha)[y]$ and for each $i=1,2,\dots, n,$ there is an integer $m_i$ and an element $v_i\in k(\alpha)$ such that $u_i=v_i y^{m_i}.$   Let $v-e_1\eta=w+a_1y+a_2y^2+\cdots+a_ly^l.$ Then we have \begin{align*}(v-e_1\eta)'&=w'+(a'_1+\beta a_1)y+\cdots+(a'_l+l\beta a_l)y^l\\
	\frac{u'_i}{u_i}&=\frac{v'_i}{v_i}+m_i\beta .\end{align*}		
Using the above equations,  we shall rewrite Equation (\ref{reduction-antiderivatives}) and obtain that
\begin{equation}\label{reduction-exponential} f= \sum_{i=1}^{n}c_i\frac{v'_i}{v_i} +e \beta +w',\end{equation} where $e=\sum^n_{i=1}m_ic_i,$ $v_i\in k(\alpha)=k(\beta)$ and $w\in k(\alpha).$

\emph{Step 4.} We shall now show that the 
 elements $w, v_1,\dots, v_n$ belong to $k$ and that $e=0.$   This will then complete the proof.

Let $\overline{E}$ be an algebraic closure of $E$ and $\bar{k}$ be the algebraic closure of $k$ in $\overline{E}.$ For any rational function $x\in k(\alpha)$ and $a\in \bar{k},$ let  \begin{equation*}
	x=r_\lambda(\alpha-a)^\lambda+r_{\lambda+1}(\alpha-a)^{\lambda+1}+\cdots
\end{equation*} be the Laurent series expansion of $x$ about $a.$ Since $$(\alpha-a)'= -R(a)-(\alpha-a)^2-(2a-r)(\alpha-a),$$
 we have the following Laurent series expansions for $x'$ and $x'/x$:

\begin{align*}
	x'&=-\lambda R(a) r_\lambda (\alpha-a)^{\lambda-1}+\cdots\\
	\frac{x'}{x}&=\lambda R(a)(\alpha-a)^{-1}+\cdots
\end{align*}

Thus, $\mathrm{ord}_a(x')\geq \mathrm{ord}_a(x)-1$ and that $\mathrm{ord}_a(x'/x)\geq -1.$ In particular, \begin{equation}\label{ordersumc_iv_i}\mathrm{ord}_a\left(\sum^n_{i=1}c_i(v'_i/v_i)\right)\geq -1.\end{equation}

Since $k$ is algebraically closed in $E,$ we have $\overline{k}E$ is a Picard-Vessiot extension of $\overline{k}$ with (see \cite[Proposition 6.6]{Magid1994}) $$\gal(\overline{k}E|\overline{k})\cong \gal(E|k)\cong \mathrm{SL}_2(C).$$  
Therefore, applying Proposition \ref{SLanalysis} with $\overline{k}$ in place of $k,$  we obtain that  $R(a)\neq 0$ for all $a\in \overline{k}.$ Thus,  if $a$ is a pole of $x$ then $\mathrm{ord}_a(x')=\mathrm{ord}_a(x)-1$ and  if $a$ is either a pole or a zero of $x$ then $\mathrm{ord}_a(x'/x)=-1.$ With these observations, we shall move on to show that $v_i$ and $w$ are in $k.$

Suppose that $w$ has a pole at $a\in \overline{k}.$ Then $\mathrm{ord}_a(w)<0$  and  as we observed, $\mathrm{ord}_a(w')=\mathrm{ord}_a(w)-1<-2.$ Note that $\mathrm{ord}_a(e\beta)=$ $\mathrm{ord}_a(e(\omega'/\omega)-e\alpha)=0$ or $1,$ depending on $e=0$ or $e\neq 0$. Therefore,  from Equations (\ref{reduction-exponential}) and (\ref{ordersumc_iv_i}), we obtain  $$0=\mathrm{ord}_a(f)\leq \min\left\{\mathrm{ord}_a\left(\sum^n_{i=1}c_i(v'_i/v_i)\right), \mathrm{ord}_a(e\beta), \mathrm{ord}_a(w')\right\}=\mathrm{ord}_a(w')<-2,$$  a contradiction.   Therefore, $w$ is a polynomial in $\alpha$ over $k.$ 

Let $A\subseteq \{1,2,\dots,n\}$ be the subset containing all $j$ such that $v_j$ has either a zero or a pole at $a\in \overline{k}.$  Suppose that $A\neq \emptyset.$ Then the Laurent series expansion of $\sum_{j\in A}c_i(v'_i/v_i)$ is 
$$\sum_{j\in A}c_i(v'_i/v_i)=\left(\sum_{i\in A}c_i\lambda_i\right) R(a)(\alpha-a)^{-1}+\cdots,$$ where $\lambda_i=\mathrm{ord}_a(v_i).$ 
Since $c_1,\dots, c_n$ are $\Q-$linearly independent, we have $\sum_{i\in A}c_i\lambda_i\neq 0.$ This implies $$\mathrm{ord}_a\left(\sum_{i\in A}c_i(v'_i/v_i)\right)=-1.$$ On the other hand, for each $i\in \{1,2,\dots, n\}\setminus A,$ we have $\mathrm{ord}_a(v'_i/v_i)\geq 0.$ Then  $$\mathrm{ord}_a\left(\sum_{i\in \{1,2,\dots,n\}\setminus A}c_i(v'_i/v_i)\right)\geq 0.$$ Thus
$$\mathrm{ord}_a\left(\sum^n_{i=1}c_i(v'_i/v_i)\right)=-1.$$ 

Since $w\in k[\alpha],$ we have  $w'\in k[\alpha]$ and that  $\mathrm{ord}_a(e\beta+w')\geq 0.$  Then $$0=\mathrm{ord}_a(f)\leq \min\left\{\mathrm{ord}_a\left(\sum^n_{i=1}c_i(v'_i/v_i)\right), \mathrm{ord}_a\left(e\beta+w'\right)\right\}=\mathrm{ord}_a\left(\sum^n_{i=1}c_i(v'_i/v_i)\right)=-1,$$ which is absurd.  Thus $A=\emptyset$ and we obtain that  $v_1,\dots, v_n$ belongs to $k$ and that $e\beta+w'=f- \sum^n_{i=1}c_i(v'_i/v_i)\in k.$   

To complete the proof, we only need to show that $e=0$ and that $w\in k.$ We have already noted that $w\in k[\alpha].$   Now we shall show that $w$ has no zeros in $\overline{k},$ which would then imply $w\in k.$    Suppose  $a\in \overline{k}$ is a zero of  $w$ of order $m\geq 1$ and that degree of $w$ is $l.$ Then $w'$ is of degree $l+1:$
\begin{align*}
	w&=a_m (\alpha-a)^m+\cdots+a_l(\alpha-a)^l , \  \  \text{where} \ a_l\neq 0, a_m\neq 0\ \ \text{and}\\
	w'&=-m a_m R(a)(\alpha-a)^{m-1}-\cdots- la_l (\alpha-a)^{l+1}.
	\end{align*}
Since $\beta=(\omega'/\omega)-2\alpha$ is a polynomial of degree $1$ and $e\beta+w'=f- \sum^n_{i=1}c_i(v'_i/v_i)\in k,$ by comparing the degrees, we obtain that $l=0.$ This in turn implies $w\in k$ and  $e=0.$  \end{proof} 
	
	\begin{lemma} \label{weierstrassian-lemma}Let $k(\theta, \theta')$ be a differential field extension of $k$, where $\theta$ is weierstrassian over $k.$  If $f\in k$ has an elementary integral over $k(\theta, \theta')$ then it has an elementary integral over $k.$
		
	\end{lemma}
	
	\begin{proof}
		Let $f\in k$ have an elementary integral over $k(\theta, \theta').$ Then, there are $\Q-$linearly independent constants $c_1,\dots, c_n,$ nonzero elements $u_1,\dots, u_n \in k(\theta,\theta')$ and an element $v\in k(\theta, \theta')$ such that 
		\begin{equation*}c_1u'_1/u_1+\cdots+c_nu'_n/u_n+v'=f\in k.\end{equation*}
			Fix an algebraic closure of $k(\theta, \theta')$ and let $\overline{k}$ be the relative algebraic closure of $k.$     Let $\gal$ be the group of all differential automorphisms of $\overline{k}(\theta,\theta')$ fixing elements of $\overline{k}.$ For $\tau\in \gal,$ we have   
		\begin{equation*}c_1\tau(u_1)'/\tau(u_1)+\cdots+c_n\tau(u_n)'/\tau(u_n)+\tau(v)'=f\in k.\end{equation*}
		From Proposition \ref{differentialdependence}, we see that there is a constant $c_\tau\in C$ such that \begin{equation}\label{sigma+elementaryexpression} c_1\mathrm{d}\tau(u_1)/\tau(u_1)+\cdots+c_n\mathrm{d}\tau(u_n)/\tau(u_n)-c_\tau \left(c_1\mathrm{d}u_1/u_1+\cdots+c_n\mathrm{d}u_n/u_n\right)+\mathrm{d}(\tau(v)-c_\tau v)=0.\end{equation}
		We extend  the set $\{c_1,\dots,c_n\}$ to a $\Q-$linearly independent set of constants $\{c_1,\dots,c_n,\dots, c_l\}$ so that for each $1\leq i\leq n,$ $c_\tau c_i$ is a $\Q-$linear combination of $c_1,\dots, c_l.$ Now for $1\leq i\leq n$ and $1\leq j\leq l,$ we can find integers $m_{ij\tau}$ and a positive integer $l_\tau$ such that  $$-l_\tau c_\tau c_i=\sum^l_{j=1}m_{ij\tau }c_j.$$ 
		Then Equation (\ref{sigma+elementaryexpression}) becomes
		\begin{equation}\label{simplifiedsigma+elementaryexpression} c_1\mathrm{d}v_1/v_1+\cdots+c_l \mathrm{d}v_l/v_l+\mathrm{d}(l_\tau v_0)=0,\end{equation} where for
		$1\leq j\leq n,$ we have $$v_j=\frac{\tau(u_j^{l_\tau})}{\prod^n_{i=1}u^{m_{ij\tau}}_i}$$ and  for each $n+1 \leq i\leq l,$  $v_i$ is a power product of  $u_1,\dots, u_n, \tau(u_1),\dots, \tau(u_n)$ and $v_0=\tau(v)-c_\tau v.$ 
		
		Apply  Proposition \ref{zerodifferential} to Equation (\ref{simplifiedsigma+elementaryexpression}) to obtain that each $v_0,v_1,\dots, v_n$  belongs to $\overline{k}.$  Thus, for each $\tau\in \gal$ and for each $1\leq j\leq n,$ there are elements $f_{j\tau}$ and $g_\tau$ in $\overline{k}$ such that \begin{equation}\label{powerproduct-translate}\tau(u_j^{l_\tau})=f_{j\tau} \prod^n_{i=1}u^{m_{ij\tau}}_i\quad\text{and} \quad \tau(v)=c_\tau v+g_\tau.\end{equation} We claim that these equations hold only if each $u_1,\dots, u_n, v$ belongs to $k,$ which will then complete the proof of the lemma. As noted earlier,  $k$ is algebraically closed in $k(\theta,\theta').$ Since each $u_1,\dots, u_n, v$ belongs to $k(\theta, \theta'),$ the claim follows once we show that  each $u_1,\dots, u_n, v$ belongs to $\overline{k}.$  
		
		Suppose that for some $j$ we have $u_j\in \overline{k}(\theta,\theta')$ and $u_j\notin\overline{k}.$   Since $\mathcal E=ZY^2-4X^3+g_1Z^2X+g_0Z^3$ is an irreducible nonsingular projective curve, every non-constant rational function in $\overline{k}(\mathcal E)=\overline{k}(\theta,\theta')$ admits a zero and a pole.  Let $$T_j:=\{y\in \mathcal E(\overline{k})\ | \ y \text{ is a pole of }\ \tau(u_j)\ \text{ for some }\tau\in \gal\}.$$ Recall that $\gal$  consists of  automorphisms induced by the  translation maps $\tau_p$ for $p\in \mathcal E(C).$ Observe that  $\{\tau_p(x)=x+p\ | \ p\in \mathcal E(C)\}$ is an infinite set and that $x\in\mathcal E(\overline{k})$ is a pole of $u_j$ if and only if for each $p\in \mathcal E(C),$ $\tau_{-p}(x)=x-p$ is a pole of $\tau^*_p(u_j).$ Thus  $T_j$ is an infinite set.   
		
		For each $1\leq i\leq n,$ let $S_i$ be the finite set of all zeros and poles of $u_i.$ Since  $f_{j\tau}\in \overline{k}$ and $m_{ij\tau}$ are integers, the set of all poles of $f_{j\tau} \prod^n_{i=1}u^{m_{ij\tau}}_i$ is a subset of the finite set $\cup^n_{i=1} S_i.$ Since $l_\tau$ is a positive integer,  if $y\in T_j$  is a pole of $\tau(u_j)$ then $y$ is also a pole of $\tau(u^{l_\tau}_j).$ Therefore,  from Equation (\ref{powerproduct-translate}), we must have $T_j\subset \cup^n_{i=1} S_i,$ which contradicts the fact that $T_j$ is an infinite set.  Thus each $u_1,\cdots, u_n$ must  belong to $\overline{k}.$ Similarly, one shows that $v\in \overline{k}.$
	\end{proof}
	
\begin{proof}[Proof of Theorem \ref{IFT-elliptic}]  By an induction on $m,$ we shall assume that $f\in k$ has an elementary integral over $E_1.$  Depending on the type of the extension $E_1,$ we  shall apply \cite[Theorem]{Rosenlicht1968} or one of the  Lemmas \ref{dihedral-lemma}, \ref{sl_2-lemma}, \ref{weierstrassian-lemma} and  obtain that $f$ has an elementary integral over $k.$ \end{proof}

	\bibliographystyle{amsalpha} 
	\bibliography{KS}

\end{document}